\documentclass[12pt,twoside,reqno]{amsart}
\usepackage[all]{xy}   
\usepackage {amssymb,latexsym,amsthm,amsmath}
\usepackage[colorlinks, linkcolor=blue, citecolor=red]{hyperref}
\usepackage{graphicx}
\usepackage{mathtools}
\usepackage{xcolor}
\usepackage[utf8]{inputenc}
\usepackage{tikz}

\long\def\symbolfootnote[#1]#2{\begingroup%
\def\thefootnote{\fnsymbol{footnote}}\footnote[#1]{#2}\endgroup}

\newcommand{\Q}{\mathbb Q}

\newcommand{\Aut}{\textup{Aut}}

\newcommand{\Id}{\textup{Id}}
\newcommand{\ad}{\mathrm{ad}}

\newcommand{\Int}{\mathrm{Int}}
\def \exp {\mathrm{exp}}
\def \U {\mathrm{U}}

\newcommand{\sslash}{//}
\newcommand{\diag}{\textup{diag}}

\newcommand{\secref}[1]{Section~\ref{#1}}
\newcommand{\thmref}[1]{Theorem~\ref{#1}}
\newcommand{\lemref}[1]{Lemma~\ref{#1}}
\newcommand{\remref}[1]{Remark~\ref{#1}}
\newcommand{\propref}[1]{Proposition~\ref{#1}}
\newcommand{\corref}[1]{Corollary~\ref{#1}}

\makeatletter
\def\imod#1{\allowbreak\mkern10mu({\operator@font mod}\,\,#1)}
\makeatother

\newtheorem{theorem}{Theorem}[section]
\newtheorem{lemma}[theorem]{Lemma}
\newtheorem{corollary}[theorem]{Corollary}
\newtheorem{proposition}[theorem]{Proposition}
\newtheorem*{theorem*}{Theorem}

\theoremstyle{definition}
\newtheorem{remark}[theorem]{Remark}
\newtheorem{notation}[theorem]{Notation}
\newtheorem{question}[theorem]{Question}

\numberwithin{equation}{section}

\newcommand{\ignore}[1]{}

\newcommand{\mynote}[1]{}
\def \GL {\mathrm{GL}}
\def \SL {\mathrm{SL}}

\def \B {\mathrm{B}}
 
\begin{document}
\setcounter{section}{0}
\setcounter{tocdepth}{1}
\title[Twisted Conjugacy in Linear Algebraic Groups]{Twisted Conjugacy in Linear Algebraic Groups}
\author[Sushil Bhunia]{Sushil Bhunia}
\author[Anirban Bose]{Anirban Bose}
\email{sushilbhunia@gmail.com}
\email{anirban.math@gmail.com}
\thanks{Bose is supported by DST-INSPIRE Faculty fellowship(IFA DST/INSPIRE/04/2016/001846)}
\address{Indian Institute of Science Education and Research (IISER) Mohali, Knowledge City,  Sector 81, S.A.S. Nagar 140306, Punjab, India}
\date{\today}
\begin{abstract}
Let $k$ be an algebraically closed field, $G$ a linear algebraic group over $k$ and $\varphi\in \Aut(G)$, the group of all algebraic group automorphisms of $G$. Two elements $x, y$ of $G$ are said to be $\varphi$-twisted conjugate if $y=gx\varphi(g)^{-1}$ for some $g\in G$. In this paper we prove that for a connected non-solvable linear algebraic group $G$ over $k$, the number of its $\varphi$-twisted conjugacy classes is infinite for every $\varphi\in \Aut(G)$.
\end{abstract}
\maketitle
\section{Introduction}
Let $G$ be a group and $\varphi$  an endomorphism of $G$. Two elements $x, y\in G$ are said to be $\varphi$-twisted conjugate, denoted by $x\sim_{\varphi} y$, if $y=gx\varphi(g)^{-1}$ for some $g\in G$. Clearly, $\sim_{\varphi}$ is an equivalence relation on $G$. The equivalence classes with respect to this relation are called \textit{$\varphi$-twisted conjugacy classes} or \textit{Reidemeister classes} of $\varphi$. If $\varphi=\Id$, then the $\varphi$-twisted conjugacy classes are the usual conjugacy classes. Let $[x]_\varphi$ denote the $\varphi$-twisted conjugacy class containing $x\in G$ and $\mathcal{R}(\varphi):=\{[x]_\varphi\mid x\in G\}$. The cardinality of $\mathcal{R}(\varphi)$, denoted by $R(\varphi)$, is called the \emph{Reidemeister number} of $\varphi$. A group $G$ is said to have  the \textit{$R_{\infty}$-property} if $R(\varphi)$ is infinite for every automorphism $\varphi$ of $G$.

The problem of determining  groups which  have the $R_{\infty}$-property is an active area of research begun by Fel'shtyn and Hill \cite{fh94} although the study of twisted conjugacy can be traced back to the works of Gantmakher in \cite{gant}. The reader may refer to \cite{FN} and the references therein for more literature. The $R_{\infty}$-property of irreducible lattices in a connected semisimple Lie group of real rank at least $2$ has been studied by Mubeena and Sankaran (see \cite[Theorem 1]{mstrans}). Nasybullov showed that if $K$ is an integral domain of zero characteristic and $\Aut(K)$ is torsion then $\GL_n(K)$ and $\SL_n(K)$ (for $n> 2$) have the $R_{\infty}$-property (see \cite{nasy12}). 
A Chevalley group $G$ (resp. a twisted Chevalley group $G'$) over a field $K$ of characteristic zero possesses the $R_{\infty}$-property if the transcendence degree of $K$ over $\Q$ is finite, see \cite[Theorem 3.2]{FN} (resp. \cite[Theorem 1.2]{bdr}).
It is worth mentioning that a reductive linear algebraic group $G$ over an algebraically closed field $k$ of characteristic zero possesses the  $R_{\infty}$-property if the transcendence degree of $k$ over $\Q$ is finite and the radical of $G$ is a proper subgroup of $G$ (see \cite[Theorem 4.1]{FN}). The converse of the latter  holds if the group $G$ is a Chevalley group of classical type ($A_n, B_n, C_n, D_n$) as shown in \cite[Theorem 8]{nas20}, where the author proves that if $k$ has infinite transcendence degree over $\Q$ then there exists an automorphism $\varphi$ of $G$ such that $R(\varphi)=1$. However, it turns out that the automorphism thus obtained is induced by a non-trivial automorphism of $k$ and therefore $\varphi$ is not an algebraic group automorphism of $G$. This motivates the following consideration:

Let $k$ be an algebraically closed field and  $G$ a linear algebraic group defined over $k$. Let $\Aut(G)$ denote the group of all algebraic group automorphisms of $G$.  
We say that $G$ has the \emph{algebraic $R_{\infty}$-property} if $R(\varphi)=\infty$ for all $\varphi\in \Aut(G)$. Since we shall deal with only algebraic automorphisms of algebraic groups in the sequel, we call the algebraic $R_\infty$-property simply as the $R_\infty$-property of $G$. A natural question that arises is the following:
\begin{question}\label{question}
Under what conditions does $G$ have the $R_{\infty}$-property ?
\end{question}
Motivation for the present work comes from the results of Springer in \cite{springer}, where he studied the twisted conjugacy classes in simply connected semisimple algebraic groups. 
A classical result by Steinberg says that for a connected linear algebraic group $G$ over an algebraically closed field $k$ and a surjective endomorphism $\varphi$ of $G$, if $\varphi$ has a finite set of fixed points then $G=\{g\varphi(g)^{-1}\mid g\in G\}=[e]_{\varphi}$, i.e., $R(\varphi)=1$. Therefore, in this scenario, $G$ does not satisfy the $R_{\infty}$-property (see \cite[Theorem 10.1]{St}).  
For an endomorphism $\varphi$ of a simple algebraic group $G$ there exists the following dichotomy:
 (1) $\varphi$ is an automorphism. (2) $\varphi$ has a finite set of fixed points (see  \cite[Corollary 10.13]{St}). 
In an attempt to answer Question \ref{question}, we prove the following sufficient condition:
\begin{theorem*}[\thmref{mainthm}]
Let $G$ be a connected non-solvable linear algebraic group over an algebraically closed field $k$. Then $G$ possesses the $R_{\infty}$-property. 
\end{theorem*}
However, the above condition is not necessary as is evident from Example \ref{eg3} in \secref{example} of this article. Also, see \propref{solvable}. 
 
\section{Preliminaries}\label{preliminaries}
In this section we fix some notations and terminologies and recall some results which will be used throughout this paper. 
\subsection{Linear algebraic groups and Chevalley groups}\label{agcg}
There are several excellent references for this topic (for example, see \cite{hum} for linear algebraic groups and  \cite{St2}, \cite{borel} for Chevalley groups). Fix an algebraically closed field $k$. A \emph{linear algebraic group} $G$ (over $k$) is a Zariski-closed subgroup of $\GL_n(k)$ for some $n\geq 1$.
Let $G$ be a connected linear algebraic group. The \emph{radical} (denoted by $R(G)$) of $G$ is defined to be the largest closed connected solvable normal subgroup of $G$. The \emph{unipotent radical} (denoted by $R_u(G)$) of $G$ is defined as the largest closed connected unipotent normal subgroup of $G$. The group $G$ is said to be \emph{semisimple} (resp. \emph{reductive}) if $R(G)=1$ (resp. $R_u(G)=1$). A non-commutative algebraic group $G$ is called \emph{simple} if it does not have any non-trivial proper closed connected normal subgroup. 
It turns out that up to isomorphism, every semisimple algebraic group over $k$ is obtained as a Chevalley group based on $k$ (see \cite[Chapter 5]{St2}), which we now define.

 Let $\mathcal{L}$ be a complex semisimple Lie algebra, $\mathcal{H}$ a Cartan subalgebra of $\mathcal{L}$ with the associated root system $\Phi$ and a simple subsystem $\Delta$. Then with respect to the adjoint action $\ad:\mathcal{L}\longrightarrow \mathfrak{gl}(\mathcal{L})$, $\mathcal{L}$ has the decomposition $\mathcal{L}=\mathcal{H}\oplus \left(\bigoplus\limits_{\alpha\in \Phi}\mathcal{L}_\alpha\right),$ where $\mathcal{L}_\alpha=\{x\in \mathcal{L}\mid \ad(h)(x)=[h,x]=\alpha(h)x\quad \forall\; h\in \mathcal{H}\}$, for each $\alpha\in \Phi$. There exists a Chevalley basis $B=\{h_\delta,e_\alpha:\delta\in \Delta,\alpha\in \Phi\}$ of $\mathcal{L}$, where $h_\alpha\in \mathcal{H}$ is the co-root associated to $\alpha\in \Phi$ and $e_\alpha\in \mathcal{L}_\alpha$ such that $[e_\alpha,e_{-\alpha}]=h_\alpha$ for each $\alpha\in \Phi$. Let $(V,\rho)$ be a finite dimensional complex representation of $\mathcal{L}$. For each $\mu\in \mathcal{H}^\ast$ let $V_\mu:=\{v\in V\mid \rho(h)(v)=\mu(h)v\quad\forall\; h\in \mathcal{H}\}$. The \emph{weights} of the representation $\rho$ are those $\mu\in \mathcal{H}^\ast$ such that $V_\mu\neq 0$ and the additive subgroup of $\mathcal{H}^\ast$ generated by the weights of $\rho$ forms a lattice, denoted by $L_\rho$. Let $L_1$ denote the lattice generated by all the weights of all representations of $\mathcal{L}$ and $L_0$ the sublattice (of $L_1$) generated by the roots $\alpha\in\Phi$. If $\rho$ is faithful then we have $L_0\subset L_\rho\subset L_1$. Moreover, if $L^\prime$ is any lattice such that $L_0\subset L^\prime\subset L_1$, then there exists a faithful representation $(V,\pi) $ of $\mathcal{L}$ such that $L_\pi=L^\prime$.
 
 So let $(V,\pi)$ be a faithful finite dimensional complex representation of $\mathcal{L}$. If $(\mathcal{U},\varphi)$ is the universal enveloping algebra of $\mathcal{L}$ then $\pi$ canonically extends to a representation of $\mathcal{U}$ in $V$ (also denoted by $\pi$). Since $\varphi$ is injective we identify the elements of $\mathcal{L}$ with their images in $\mathcal{U}$ under the homomorphism $\varphi$. Let $\mathcal{U}_\mathbb{Z}$ be the $\mathbb{Z}$-algebra generated by the elements $e_\alpha^n/n!$ ($\alpha\in \Phi, n\geq 0$). A \emph{lattice in $V$} is a finitely generated free abelian subgroup of $V$ with a $\mathbb{Z}$-basis which is also a $\mathbb{C}$-basis of $V$. It turns out that $V$ admits a lattice $M$ (say) which is invariant under $\pi(e_\alpha)^n/n!$ ($\alpha\in \Phi, n\geq 0$) and hence under $\mathcal{U}_\mathbb{Z}$ \cite[p. 16, Corollary 1(a)]{St2}. Now consider the module $M\otimes_\mathbb{Z} \mathbb{Z}[X]$, where $X$ is an indeterminate. Then $X^n\pi(e_\alpha)^n/n!$ acts on $M\otimes_\mathbb{Z} \mathbb{Z}[X]$. Since $\pi(e_\alpha)^n$ acts as zero on $V$ for sufficiently large values of $n$, we note that the expression $\sum\limits_{n=0}^\infty X^n\pi(e_\alpha)^n/n!$ acts on $M\otimes_\mathbb{Z} \mathbb{Z}[X]$ and hence on $M\otimes_\mathbb{Z} \mathbb{Z}[X]\otimes_\mathbb{Z} k$, where $k$ is an algebraically closed field. If $V_k:=M\otimes_\mathbb{Z} k$ then for each $t\in k$ the homomorphism $M\otimes_\mathbb{Z} \mathbb{Z}[X]\otimes_\mathbb{Z} k\longrightarrow V_k$ given by $X\mapsto t$ induces an automorphism $\sum\limits_{n=0}^\infty t^n\pi(e_\alpha)^n/n!$ of the $k$-vector space $V_k$. Set $\exp(t\pi(e_\alpha)):=\sum\limits_{n=0}^\infty t^n\pi(e_\alpha)^n/n!$.
 
 A \emph{Chevalley group of type $\Phi$ over $k$} associated to $\pi$ is  defined as the group $G(\pi,k):=\langle \exp (t\pi(e_\alpha))\mid \alpha\in \Phi,t\in k\rangle$.  It is a semisimple algebraic group over $k$ and every semisimple algebraic group over $k$ (with an associated root system $\Phi$) is isomorphic to some $G(\pi, k)$. If $L_\pi=L_1$, then the associated Chevalley group (denoted by $G_{sc}$) is said to be of \emph{universal type} and it is simply connected as an algebraic group. If $L_\pi=L_0$ then the associated Chevalley group (denoted by $G_{ad}$) is said to be of \emph{adjoint type} (it is adjoint in the sense of an algebraic group over $k$). 
  
Let $G(\pi ,k)$ and $G(\rho,k)$ be Chevalley groups constructed from the same Lie algebra $\mathcal{L}$ such that $L_{\rho}\subset L_{\pi}$. Then there exists a homomorphism of algebraic groups $f:G(\pi,k)\longrightarrow G(\rho,k)$ such that $f(\exp(t\pi(e_\alpha)))=\exp(t\rho(e_\alpha))$ for all $\alpha\in \Phi,t\in k$. If $L_\pi=L_\rho$ then $f$ is an isomorphism. If $\rho$ is the adjoint representation of $\mathcal{L}$ (i.e., $L_\rho=L_0$) then $\ker(f)=Z(G(\pi,k))$ \cite[Corollary 5 (p. 29), Corollary 1 (p. 41)]{St2}.

\subsection{Automorphisms}\label{automorphisms}

Let $G$ be a connected semisimple algebraic group. Fix a maximal torus $T$ and a Borel subgroup $B$ of $G$ such that $T\subset B$. Let $\Phi$ be the root system of $G$ determined by $T$ and $\Delta$ the simple subsystem of $\Phi$ determined by $B$. Let $\Aut(G)$ denote the group of all algebraic group automorphisms of $G$ and $\Int(G)$ the group of all inner automorphisms of $G$ (hence $\Int(G)$ is a subgroup of $\Aut(G)$). If $D:=\{\varphi\in\Aut(G)\mid \varphi(T)=T,\varphi(B)=B\}$ and $\Gamma$ is the group of all automorphisms of $\Phi$ which stabilize $\Delta$, then there exists a natural homomorphism $D\longrightarrow \Gamma$. With the above notation, we quote the following result:

\begin{lemma}\cite[Theorem 27.4]{hum}\label{staut}
Let $G$ be a semisimple algebraic group. Then the following holds: 
\begin{enumerate}
\item\label{staut1} $\Aut(G)=\Int(G)D$. 
\item\label{staut2} The natural map $D\longrightarrow \Gamma$ induces a monomorphism $\Aut(G)/\Int(G)\longrightarrow \Gamma$.
\end{enumerate}
\end{lemma}

Furthermore, since we are working over an algebraically closed field, the group $G$ is split and hence we have 
\begin{lemma}\cite[Theorem 25.16]{kmrt}\label{kernel}
If $G$ is a simply connected or adjoint type semisimple algebraic group, then the sequence
\[\xymatrix{1\ar[r]& \Int(G)\ar[r]&\Aut(G)\ar[r]&\Gamma\ar[r]&1}\] is split exact. 
\end{lemma}
 
\begin{remark}\label{graphexist}
	If $G$ is simple with root system $\Phi$ then the group $\Gamma$ is non-trivial if and only if $\Phi$ is of type $A_l\, (l\geq 2), D_l\, (l\geq 4)$ and $E_6$ (see \cite[p. 66, Table 1]{humlie} and Figure \ref{Figure1}). Therefore if $\Phi$ is not of type $A_{2l}$ ($l\geq1$) then for every $\gamma\in \Gamma$ there exists a simple root $\alpha\in \Phi$ such that $\gamma(\alpha)=\alpha$.
\end{remark}
\newpage
\begin{figure}[ht]
	\caption{} \label{Figure1}
	 \vskip 0.2 cm
\centering 
\begin{tikzpicture}
\node at (-0.5,0) {$A_l\, (l\geq 2)$};
\node (a1) at (1,0) {$\bullet$};
\node (a2) at (2,0) {$\bullet$};
\node (al-1) at (3,0) {$\bullet$};
\node (al) at (4,0) {$\bullet$};
\draw[-] (1,0) -- (2,0);
\draw[dashed] (a2)-- (al-1);
\draw[-] (3,0) -- (4,0);
\draw[<->] (a1) edge[bend left=50] node[above] {} (al);
\draw[<->] (a2) edge[bend left=60] node[above] {} (al-1);
\node at (7,0) {$\Gamma\cong\mathbb{Z}/2\mathbb{Z}$};
\end{tikzpicture}

\begin{tikzpicture}
\node at (-1,0) {$D_4$};
\node (d1) at (1,0) {$\bullet$};
\node (d2) at (2,0) {$\bullet$};
\node (d3) at (2.5, 0.7) {$\bullet$};
\node (d4) at (2.5,-0.7) {$\bullet$};
\draw[-] (1,0) -- (2,0);
\draw[-] (2,0) -- (2.5,0.7);
\draw[-] (2,0) -- (2.5,-0.7);
\draw[<->] (d1) edge[bend left=50] node[above] {} (d3);
\draw[<->] (d1) edge[bend right=50] node[above] {} (d4);
\draw[<->] (d3) edge[bend left=50] node[above] {} (d4);
\node at (7,0) {$\Gamma\cong S_3$};
\end{tikzpicture}

\begin{tikzpicture}
\node at (-0.5,0) {$D_l\, (l> 4)$};
\node (d1) at (1,0) {$\bullet$};
\node (d2) at (2,0) {$\bullet$};
\node (dl-3) at (3,0) {$\bullet$};
\node (dl-2) at (4,0) {$\bullet$};
\node (dl-1) at (5, 0.5) {$\bullet$};
\node (dl) at (5,-0.5) {$\bullet$};
\draw[-] (1,0) -- (2,0);
\draw[dashed] (d2) -- (dl-3);
\draw[-] (3,0) -- (4,0);
\draw[-] (4,0) -- (5,0.5);
\draw[-] (4,0) -- (5,-0.5);
\draw[<->] (dl-1) edge[bend left=60] node[above] {} (dl);
\node at (7,0) {$\Gamma\cong\mathbb{Z}/2\mathbb{Z}$};
\end{tikzpicture}

\begin{tikzpicture}
\node at (-1,0) {$E_6$};
\node (e1) at (1,0) {$\bullet$};
\node (e2) at (2,0) {$\bullet$};
\node (e3) at (3,0) {$\bullet$};
\node (e4) at (4,0) {$\bullet$};
\node (e5) at (5,0) {$\bullet$};
\node (e6) at (3,1) {$\bullet$};
\draw[-] (1,0) -- (2,0);
\draw[-] (2,0) -- (3,0);
\draw[-] (3,0) -- (4,0);
\draw[-] (4,0) -- (5,0);
\draw[-] (3,0) -- (3,1);
\draw[<->] (e1) edge[bend right=50] node[above] {} (e5);
\draw[<->] (e2) edge[bend right=50] node[above] {} (e4);
\node at (7,0) {$\Gamma\cong\mathbb{Z}/2\mathbb{Z}$};
\end{tikzpicture}
\end{figure}
\subsection{Some basic results}\label{general}
The contents of this section are known for any abstract group. For the sake of completeness, we prove the following results in the context of algebraic groups. The proofs essentially follow the lines of argument used in \cite[Corollary 3.2]{FLT} and \cite[Lemma 2.1, Lemma 2.2]{MS}.
\begin{lemma}\label{inner}
	Let $\varphi\in \Aut(G)$ and $\Int_g$ the inner automorphism defined by   $g\in G$. Then $R(\varphi\circ\Int_g)=R(\varphi)$. In particular, $R(\Int_g)=R(\Id)$, i.e., the number of inner twisted conjugacy classes in $G$ is equal to the number of conjugacy classes in $G$.
\end{lemma}
\begin{proof}
	Let $x,y\in G$ such that $[x]_{\varphi\circ\Int_g}=[y]_{\varphi\circ\Int_g}$. Then there exists a $z\in G$ such that
	$y=zx(\varphi\circ\Int_g)(z^{-1})=zx\varphi(gz^{-1}g^{-1})=zx\varphi(g)\varphi(z^{-1})\varphi(g^{-1})$. This implies that 
	$y\varphi(g)=zx\varphi(g)\varphi(z^{-1})$, i.e, $[x\varphi(g)]_{\varphi}=[y\varphi(g)]_{\varphi}$. 
	Thus we get a well-defined map \[\widehat{\varphi}: \mathcal{R}(\varphi\circ \Int_g)\longrightarrow \mathcal{R}(\varphi)\] given by  $\widehat{\varphi}([x]_{\varphi\circ \Int_g})=[x\varphi(g)]_{\varphi}$. This is bijective as well. Therefore $R(\varphi\circ\Int_g)=R(\varphi)$. Second part of the lemma follows by taking $\varphi=\Id$.
\end{proof}

\begin{lemma}\label{ses}
	Suppose that $1\longrightarrow N \overset{i}{\longrightarrow} G\overset{\pi}{\longrightarrow} Q\longrightarrow 1$ is  an exact sequence of algebraic groups. Let $\varphi\in\Aut(G)$ be such that $\varphi(N)=N$. Let $\overline{\varphi}$ denote the automorphism of $Q$ induced by $\varphi$. Then
	\begin{enumerate}
		\item \label{qtog} $R(\varphi)\geq R(\overline{\varphi})$.  
		\item \label{gtoq}
		If $N$ is finite and $R(\varphi)=\infty$ then $R(\overline{\varphi})=\infty$.
		\item \label{ntog} If $Q$ is finite and $R(\varphi|_N)=\infty$ then $R(\varphi)=\infty$.
	\end{enumerate}	
\end{lemma}
\begin{proof}
	(1) Since $\varphi\in \Aut(G)$ and $\varphi(N)=N$, then the restriction of $\varphi$ to $N$, denoted by $\varphi|_{N}$, is an automorphism of $N$. Clearly, $\varphi$ induces an automorphism of $Q\cong G/N$, denoted by $\overline{\varphi}$ such that the  following diagram commutes:
	\[\xymatrix{
		1\ar[r] & N\ar[r]^{i}\ar[d]_{\varphi|_{N}}& G\ar[r]^{\pi} \ar[d]^{\varphi}& Q\ar[r]\ar[d]^{\overline{\varphi}}&1 \\
		1\ar[r] & N\ar[r]^{i}& G\ar[r]^{\pi}& Q\ar[r]&1 
	}.\] In particular, we have $\overline{\varphi}\circ \pi=\pi\circ \varphi$. Now, observe that $\pi$ induces a surjective map $\widehat{\pi}:\mathcal{R}(\varphi)\longrightarrow\mathcal{R}(\overline{\varphi})$ given by $\widehat{\pi}([x]_{\varphi})=[\pi(x)]_{\overline{\varphi}}$ for all $x\in G$. Therefore $R(\varphi)\geq R(\overline{\varphi})$. 
	
	(2) If possible suppose that not all fibres of $\widehat{\pi}$ are finite. Let $x_r\in G$ ($r\geq 0$) be such that $\widehat{\pi}([x_r]_{\varphi})=\widehat{\pi}([x_0]_{\varphi})$ for all $r\geq 1$ but $[x_r]_{\varphi}\neq [x_s]_{\varphi}$ for all $r\neq s$ (this is possible bacause $R(\varphi)=\infty$). Then for each $r\geq 1$, we have 
	\[\pi(x_0)=\pi(g_r)\pi(x_r)\overline{\varphi}(\pi(g_r))^{-1}=\pi(g_r)\pi(x_r)\pi(\varphi(g_r))^{-1}=\pi(g_rx_r\varphi(g_r)^{-1}),\]
	for some $g_r\in G$. Thus, for each $r\geq 1$ there exists $n_r\in N$ such that $x_0n_r=g_rx_r\varphi(g_r)^{-1}$. Hence $[x_0n_r]_{\varphi}=[x_r]_{\varphi}$ for all $r$, which is a contradiction, since $N$ is finite and $[x_r]_{\varphi}\neq [x_s]_{\varphi}$ for $r\neq s$. Hence all the fibres of $\widehat{\pi}$ are finite whenever $N$ is finite. Therefore, by the previous part of this lemma and our assumption that $R({\varphi})=\infty$, we have $R(\overline{\varphi})=\infty$.
	
	(3) If possible suppose that $R(\varphi)<\infty$. Since $R(\varphi|_{N})=\infty$ then there exist infinitely many elements $n_r\in N$ ($r\geq 0$) such that $[n_r]_{\varphi|_{N}}\neq [n_s]_{\varphi|_{N}}$ ($r\neq s$) but $[n_r]_{\varphi}=[n_0]_{\varphi}$ for all $r\geq 0$. So for all $r\geq 1$ there exists $x_r\in G$ such that $n_r=x_rn_0\varphi(x_r^{-1})$. 
	Again since $Q(\cong G/N)$ is finite then there exist distinct $r,s\in \mathbb{N}$ such that $x_sx_r^{-1}\in N$. Therefore $$n_s=x_sn_0\varphi(x_s^{-1})=x_sx_r^{-1}n_r\varphi(x_r)\varphi(x_s^{-1})=(x_sx_r^{-1})n_r\varphi(x_sx_r^{-1})^{-1}.$$ Hence $[n_r]_{\varphi|_N}=[n_s]_{\varphi|_N}$, which is a contradiction. 
	
	This completes the proof.
\end{proof}

\section{Proof of the main result}\label{mainsec}
The following results will be used in the proof of our main theorem.  

\begin{lemma}\label{autprod}
Suppose that $G$ is a simple algebraic group. Then 
\[\Aut(G^n)\cong (\Aut(G))^n\rtimes S_n.\]
\end{lemma}
\begin{proof}
Let $G^n=G_1\times G_2\times \cdots \times G_n$, where $G_i=G$ and set $H_i=(1\times \cdots \times G_i\times \cdots \times 1)$ for $i\in I:=\{1,...,n\}$. Then $G^n$ is semisimple and $H_1,...,H_n$ are its simple components (i.e., minimal closed connected normal subgroups of $G$ of positive dimension). If $\varphi\in \Aut(G^n)$ then for every $i\in I,$  $\varphi(H_i)$ is also a simple component and hence there exists $k_i\in I$ such that $\varphi(H_i)=H_{k_i}$ (by Theorem 27.5 of \cite{hum}). Therefore, $\varphi$ gives rise to a bijection $\pi(\varphi):I\longrightarrow I$ defined by $\pi(\varphi)(i)=k_i$ for all $i\in I$. Hence, we have a homomorphism $\pi:\Aut(G^n)\longrightarrow S_n$ such that $\ker(\pi)=\Aut(G)^n$, where $\Aut(G)^n$ is identified with a subgroup of  $ \Aut(G^n)$ via
\[(\varphi_1, \ldots, \varphi_n)\mapsto \varphi_{(1, \ldots, n)}((x_1, \ldots, x_n)\mapsto (\varphi_1(x_1), \ldots, \varphi_n(x_n)).\]

\noindent Now, for every $\gamma\in S_n$, define $\overline{\gamma}:G^n\longrightarrow G^n$ by $\overline{\gamma}(x_1, \ldots, x_n)=(x_{\gamma^{-1}(1)}, \ldots, x_{\gamma^{-1}(n)})$ for all $(x_1,...,x_n)\in G^n$ and observe that $\gamma\mapsto \overline{\gamma}$ defines an injective homomorphism $S_n \hookrightarrow \Aut(G^n)$. Observe that $\overline{\gamma}(H_i)=H_{\gamma(i)}$ and hence $\pi(\overline{\gamma})(i)=\gamma(i)$ for all $i\in I, \gamma\in S_n$.
This shows that $\pi$ is surjective and the exact sequence $$1\longrightarrow \Aut(G)^n\longrightarrow \Aut(G^n)\overset{\pi}\longrightarrow S_n\longrightarrow 1$$ splits on the right. Hence $\Aut(G^n)\cong \Aut(G)^n\rtimes S_n$ as desired.
\end{proof}
We get the following characterization of the $R_{\infty}$-property using the above lemma. 
\begin{lemma}\label{productgn}
Suppose that a simple algebraic group $G$  possesses the $R_{\infty}$-property. Then $G^n$ also possesses the $R_{\infty}$-property.
\end{lemma}
\begin{proof}
Let $\varphi\in \Aut(G^n)$. Then by \lemref{autprod}, we have $\varphi=(\varphi_1, \varphi_2, \ldots, \varphi_n; \overline{\sigma})$, where $\varphi_i\in \Aut(G)$ for all $i=1, 2, \ldots, n$ and $\sigma\in S_n$. If $\sigma^{-r}(1)=1$ for some smallest positive integer $r$, then choose $a_i$'s from distinct $\varphi_1\varphi_{\sigma^{-1}(1)}\cdots \varphi_{\sigma^{-(r-1)}(1)}$-twisted conjugacy classes in $G$ ($i\in\mathbb{N}$). 

\textbf{Claim:} $R(\varphi)=\infty$. 

If possible let $R(\varphi)<\infty$. Then without loss of generality we may assume that $(a_i, 1, \ldots, 1)\sim_{\varphi} (a_1, 1, \ldots, 1)$ in $G^n$ for $i=2, 3, \ldots$. Then for some fixed $i$ there exists $\widetilde{g}_i\in G^n$ such that 
\begin{align}\label{eq1}
(a_1, 1, \ldots, 1)=\widetilde{g}_i(a_i, 1, \ldots, 1)\varphi(\widetilde{g}_i)^{-1}
\end{align}
where $\widetilde{g}_i=(g_1, g_2, \ldots, g_n)\in G^n$. Now
\begin{align}\label{eq2}
\varphi(\widetilde{g}_i)=(\varphi_1, \varphi_2, \ldots, \varphi_n; \overline{\sigma})(g_1, g_2, \ldots, g_n)=(\varphi_1(g_{\sigma^{-1}(1)}), \varphi_2(g_{\sigma^{-1}(2)}), \ldots, \varphi_n(g_{\sigma^{-1}(n)})). 
\end{align}
Therefore from Equations \eqref{eq1} and \eqref{eq2}, we get 
\begin{align*}
(a_1, 1, \ldots, 1)=(g_1, g_2, \ldots, g_n)(a_i, 1, \ldots, 1)(\varphi_1(g_{\sigma^{-1}(1)}^{-1}), \varphi_2(g_{\sigma^{-1}(2)}^{-1}), \ldots, \varphi_n(g_{\sigma^{-1}(n)}^{-1})).
\end{align*}
This implies 
\begin{align*}
a_1&=g_1a_i\varphi_1(g_{\sigma^{-1}(1)}^{-1}),\\ 
g_j^{-1}&=\varphi_j(g_{\sigma^{-1}(j)}^{-1})
\end{align*}
for $j=2, 3, \ldots, n$. Therefore, from the above two equations, we get 
\begin{align*}
a_1&=g_1a_i\varphi_1(g_{\sigma^{-1}(1)}^{-1})=g_1a_i\varphi_1\varphi_{\sigma^{-1}(1)}(g_{\sigma^{-2}(1)}^{-1})\\
&\vdots\\
&=g_1a_i(\varphi_1\varphi_{\sigma^{-1}(1)}\cdots \varphi_{\sigma^{-(r-1)}(1)})(g_{\sigma^{-r}(1)}^{-1})\\
&=g_1a_i(\varphi_1\varphi_{\sigma^{-1}(1)}\cdots \varphi_{\sigma^{-(r-1)}(1)})(g_1^{-1}).
\end{align*}
Which is a contradiction to the fact that $a_i$'s are in different $\varphi_1\varphi_{\sigma^{-1}(1)}\cdots \varphi_{\sigma^{-(r-1)}(1)}$-twisted conjugacy classes in $G$. Therefore $R(\varphi)=\infty$. 
\end{proof}

\begin{notation}\label{not}
Let $G=G(\pi, k)$ be a Chevalley group of type $\Phi$ over $k$.
For each $\alpha\in \Phi$ set  
\begin{align*}
x_{\alpha}(t)&=\exp(t\pi(e_{\alpha})) \quad t\in k\\
X_{\alpha}&=\langle x_{\alpha}(t)\mid t\in k\rangle\\
n_{\alpha}(t)&=x_{\alpha}(t)x_{-\alpha}(-t^{-1})x_{\alpha}(t) \quad t\in k^{\times} \\
h_{\alpha}(t)&=n_{\alpha}(t)n_{\alpha}(-1) \quad t\in k^{\times}.
\end{align*}  
Then $G=\langle x_{\alpha}(t)\mid\alpha\in \Phi, t\in k\rangle$. Viewing $G$  as a semisimple linear algebraic group over $k$, we observe that $T=\langle h_{\alpha}(t)\mid\alpha\in \Phi, t\in k^{\times}\rangle$ is a maximal torus of $G$ and $B=T\ltimes U$ is a Borel subgroup of $G$, where $U=\langle x_{\alpha}(t)\mid \alpha\in \Phi^+, t\in k\rangle$ and $\Phi^+$ a positive subsystem. With the above notation we state 

\begin{lemma}\cite[Corollary 6 (p. 31), Corollary 2 (p. 41)]{St2}\label{onedim}
 For every $\alpha\in \Phi$ there exists a homomorphism of algebraic groups $\varphi_\alpha:\SL_2(k)\longrightarrow \langle X_\alpha,X_{-\alpha}\rangle$ such that $\varphi_\alpha\begin{pmatrix}
1&t\\0&1
\end{pmatrix}=x_\alpha(t)$, $\varphi_\alpha\begin{pmatrix}
1&0\\t&1\end{pmatrix}=x_{-\alpha}(t)$ (for all $t\in k$), $\varphi_\alpha\begin{pmatrix}
0&1\\-1&0
\end{pmatrix}=n_\alpha(1)$, and $\varphi_\alpha\begin{pmatrix}
t&0\\0&t^{-1}
\end{pmatrix}=h_\alpha(t)$ (for all $t\in k^\times$), Moreover, $\ker(\varphi_\alpha)=\{1\}$ or $\{\pm 1\}$,  i.e., $\langle X_\alpha,X_{-\alpha}\rangle$ is either $\SL_2(k)$ or $\mathrm{PSL}_2(k)$.
\end{lemma}
\begin{lemma}\cite[p. 29, Corollary (a) to Lemma 28]{St2}\label{onedim1}
 If $G$ is of universal type (i.e., $L_\pi=L_1$) then every element $h$ of $T$ can be uniquely written as $h=\prod\limits_{i=1}^nh_{\alpha_i}(t_i),$ $t_i\in k^\times$, where $\Delta=\{\alpha_1,\ldots, \alpha_n\}$ is a simple subsystem of $\Phi$, determined by $B$. 
 \end{lemma}
Let $\Gamma$ be as in \secref{automorphisms} and $\rho\in \Gamma$. Then by Theorem 29 of \cite[p. 90]{St2}, there exists an abstract automorphism $\overline{\rho}:G\longrightarrow G$ such that $\overline{\rho}(x_{\alpha}(t))=x_{\rho(\alpha)}(\epsilon_{\alpha}t)$ for all $\alpha\in \Phi, t\in k$, where $\epsilon_\alpha$ is a sign depending on $\alpha\in \Phi$. Furthermore, $\epsilon_{\alpha}=1$ if $\alpha$ or $-\alpha$ is a simple root.
\end{notation}
\begin{lemma}\label{algebraic}
For every $\rho\in \Gamma$, the map $\overline{\rho}$  is an algebraic group automorphism of $G$. Moreover, $\overline{\rho}\in D$ (where $D$ is as in \secref{automorphisms}).
\end{lemma}
\begin{proof}
We observe that for every $\alpha\in\Phi$, 
\begin{align*}
\overline{\rho}(h_{\alpha}(t))&=\overline{\rho}\left(n_\alpha(t)n_\alpha(-1)\right)=\overline{\rho}\left(x_\alpha(t)x_{-\alpha}(-t^{-1})x_\alpha(t)x_\alpha(-1)x_{-\alpha}(1)x_\alpha(-1)\right)\\
&=x_{\rho(\alpha)}(\epsilon_{\alpha}t)x_{\rho(-\alpha)}(-\epsilon_{-\alpha}t^{-1})x_{\rho(\alpha)}(\epsilon_{\alpha}t)x_{\rho(\alpha)}(-\epsilon_{\alpha})x_{\rho(-\alpha)}(\epsilon_{-\alpha})x_{\rho(\alpha)}(-\epsilon_{\alpha})\\
&=n_{\rho(\alpha)}(\epsilon_{\alpha}t)n_{\rho(\alpha)}(-\epsilon_{\alpha})\quad (\text{since } \rho(-\alpha)=-\rho(\alpha), \epsilon_{\alpha}\epsilon_{-\alpha}=1)\\
&=h_{\rho(\alpha)}(\epsilon_{\alpha}t). 
\end{align*} 
Therefore $\overline{\rho}(T)=T$. Also, note that $\overline{\rho}(B)=B$, since $\overline{\rho}(U)=U$. It remains to show that $\overline{\rho}$ is an algebraic morphism. In view of \cite[Lemma 32.1]{hum}, it is sufficient to prove that the restriction of $\overline{\rho}$ on the Borel subgroup $B$ of $G$ is a morphism. Now, observe that $X_{\alpha}\longrightarrow X_{\rho(\alpha)}$ (defined by $x_\alpha(t)\mapsto x_{\rho(\alpha)}(t), t\in k$) and the product map $X_{\alpha_1}\times X_{\alpha_2}\times \cdots \times X_{\alpha_r}\longrightarrow U$, defined by $(x_{\alpha_1}(t_1),\ldots, x_{\alpha_r}(t_r))\mapsto \prod x_{\alpha_i}(t_i)$ ($t_i\in k, i=1,\ldots,r$) are isomorphisms of varieties, where $\Phi^+=\{\alpha_1, \ldots, \alpha_r\}$ is the set of all positive roots (taken in any order). Therefore $\overline{\rho}$ is an isomorphism of varieties when restricted to $U$. Also, we have the following commutative diagram:
\[\xymatrix{\mathbb{G}_m^n \ar[r]^{f} \ar[dr]_{\overline{\rho}\circ f}  
	& T \ar[d]^{\overline{\rho}} \\
	&T },\] where $f(t_1, \ldots, t_n)=\prod_{i=1}^nh_{\alpha_i}(t_i)$ for $t_i\in k^\times$ and $\{\alpha_1,...,\alpha_n\}$ is the set of all simple roots. This in turn induces 
\[\xymatrix{\frac{\mathbb{G}^n_m}{K} \ar[r]^{\widetilde{f}} \ar[dr]_{\widetilde{\overline{\rho}\circ f}}  
	& T \ar[d]^{\overline{\rho}} \\
	&T }\] where $K:=\ker(f)=\ker(\overline{\rho}\circ f)$, and $\widetilde{f}$ and $\widetilde{\overline{\rho}\circ f}$ are isomorphism of varieties. Therefore, $\overline{\rho}:T\longrightarrow T$ is a morphism. Thus, $\overline{\rho}$ is a morphism when restricted to $B=TU$. Similarly it can be shown that $\overline{\rho}^{-1}$ is also a morphism of varieties. Hence, the proof.
\end{proof}

Before moving on to prove our next set of results we make the following observation:
Let $G$ be a simply connected or adjoint type semisimple algebraic group and $\varphi\in \Aut(G)$. Then by \lemref{staut}, $\varphi=f\circ \Int_g$ for some $f\in D$ and $\Int_g\in \Int(G)$. Now let $\rho\in \Gamma$ be the automorphism of $\Phi$ induced by $f\in D$ and consider the element $\overline{\rho}\in D$ induced by $\rho$ as in \lemref{algebraic}. Since $\overline{\rho}(X_\alpha)=X_{\rho(\alpha)}$ for all $\alpha\in \Phi$, we note that $\overline{\rho}$ also induces the automorphism $\rho$ of $\Phi$. Therefore by Lemma \ref{kernel}, $\overline{\rho}^{-1}\circ f\in \Int(G)$, i.e., $f=\overline{\rho}\circ \Int_y$ for some $y\in G$ and hence, $\varphi=\overline{\rho}\circ \Int_{yg}$. So, by \lemref{inner}, we have $R(\varphi)=R(\overline{\rho})$. We summarise this as: 

\begin{lemma}\label{ets}
Let $G$ be a simply connected or adjoint type semisimple algebraic group. Then $G$ has the $R_\infty$-property if and only if $R(\overline{\rho})=\infty$ for all $\rho\in \Gamma$.
\end{lemma}

Now let $G$ be a reductive algebraic group and $\varphi\in\Aut(G)$. Then $G$ acts on itself via $\varphi$-twisted conjugacy. Let $G//_\varphi G$ be the invariant theoretic quotient. Then there exists a surjective morphism $\pi:G\longrightarrow G//_\varphi G$ which is constant on the orbits in $G$ (cf. \cite[Theorem 3.5]{new78}), i.e., if $x\sim_\varphi y$ for $x,y\in G$, then $\pi(x)=\pi(y)$. Therefore if $[x]_\varphi$ denotes the orbit of $x\in G$, then the assignment $[x]_\varphi\mapsto \pi(x)$ (for all $x\in G$) is well-defined and onto. We record this as 
\begin{lemma}\label{git}
Let $G$ be a reductive algebraic group and $\varphi\in \Aut(G)$. 
Then there exists a surjection $ \mathcal{R}(\varphi)\longrightarrow G\sslash_{\varphi}G$, where $\mathcal{R}(\varphi)$ is the set of all $\varphi$-twisted conjugacy classes of $G$.
\end{lemma}

\begin{proposition}\label{step1}
Every simply connected simple algebraic group $G$ possesses the $R_{\infty}$-property. 
\end{proposition}
\begin{proof}
Let $\Phi$ be the irreducible root system of $G$ determined by a maximal torus $T$.
 In view of \lemref{ets}, it suffices to show that $R(\overline{\rho})=\infty$ for every $\rho\in \Gamma$. First, let $\rho=1\in \Gamma$. Then by \lemref{kernel}, $\overline{\rho}=\Int_y$ for some $y\in G$. Therefore in view of \lemref{inner}, we have $R(\overline{\rho})=R(\Int_y)=R(\Id)$ which is the number of conjugacy classes in $G$. Since  there are infinitely many semisimple conjugacy classes in $G$ (as $k$ is algebraically closed), we conclude that $R(\overline{\rho})=\infty$.
 
  Now assume that $\rho\neq 1$. Then $\Phi$ is of type $A_l \;(l\geq 2), D_l \;(l\geq 4)$ or $E_6$ (cf. \remref{graphexist}).
 By \cite[Corollary 2]{springer}, we have $G\sslash_{\overline{\rho}}G\cong \mathbb{A}^d$, where $d=\dim T^{\overline{\rho}}$ and $T^{\overline{\rho}}=\{t\in T\mid \overline{\rho}(t)=t\}$. By \lemref{git} it is enough to prove that $G\sslash_{\overline{\rho}}G$ is infinite, i.e., $d\geq1$. We do this by considering the following two cases:
 
 \textbf{Case I:} Assume that $\Phi$ is of type $A_{2l+1} (l\geq 1), D_l(l\geq 4)$ or $E_6$. By \remref{graphexist} there exists a simple root (say) $\alpha\in \Phi$, such that $\rho(\alpha)=\alpha$. For each such  $\alpha$, we note that $\overline{\rho}(h_{\alpha}(t))=h_{\rho(\alpha)}(t)=h_{\alpha}(t)$ for all $t\in k^{\times}$ (cf. proof of \lemref{algebraic}). Thus $h_{\alpha}(t)\in T^{\overline{\rho}}$ for all $t\in k^{\times}$. Hence $\{h_\alpha(t)\mid t\in k^\times\}$ is a $1$-dimensional torus (by \lemref{onedim}) contained in $T^{\overline{\rho}}$. Therefore $d=\dim(T^{\overline{\rho}})\geq 1$. 

\textbf{Case II:} Let $\Phi$ be of type $A_{2l}$ ($l\geq 1$). Let $\alpha\in \Phi$ be a simple root. Then $\rho(\alpha)\neq \alpha$. Since $\rho^2=1$ we have  $\overline{\rho}(h_{\alpha}(t)h_{\rho(\alpha)}(t))=h_{\rho(\alpha)}(t)h_{\rho^2(\alpha)}(t)=h_{\rho(\alpha)}(t)h_{\alpha}(t)=h_{\alpha}(t)h_{\rho(\alpha)}(t)$ for all $t\in k^{\times}$. Thus $S:=\{h_{\alpha}(t)h_{\rho(\alpha)}(t)\mid t\in k^\times\}\subset T^{\overline{\rho}}$. Following \lemref{onedim}, we have a homomorphism $\psi:D_2(k)\longrightarrow T$ be defined by $\psi\begin{pmatrix}
t&0\\0&t^{-1}
\end{pmatrix}=\varphi_{\alpha}\begin{pmatrix}
t&0\\0&t^{-1}\end{pmatrix}
\varphi_{\rho(\alpha)}\begin{pmatrix}
t&0\\0&t^{-1}\end{pmatrix}$ 
($t\in k^\times$), where $D_2(k)$ is the diagonal maximal torus of $\SL_2(k)$. Note that  $\psi(D_2(k))=S$, which implies that $S$ is a connected subgroup of $T^{\overline{\rho}}$. Hence $S$ is a torus. We \emph{claim} that $S$ is non-trivial. So if possible let $S$ be trivial, i.e., $h_{\alpha}(t)h_{\rho(\alpha)}(t)=1$ for all $t\in k^\times$. But then by \lemref{onedim1} $h_\alpha(t)=h_{\rho(\alpha)}(t)=1$ for all $t\in k^\times$, a contradiction to the fact that $\{h_\alpha(t)\mid t\in k^\times\}$ is a $1$-dimensional torus. Hence, the claim. Therefore $\dim S=1$ which implies that $d=\dim T^{\overline{\rho}}\geq 1$ as desired.  

This completes the proof.
\end{proof}

\begin{corollary}\label{step2}
 Every simple algebraic group $G$ possesses the $R_{\infty}$-property.
\end{corollary}
\begin{proof}
Let $\Phi$ be the irreducible root system associated to $G$. Following \secref{agcg} $G$ can be treated as a Chevalley group of type $\Phi$ over $k$. So let $G=G(\pi,k)$ and $G_{sc}=G(\pi_1,k)$ (resp. $G_{ad}=G(\ad,k)$) a universal (resp. adjoint) Chevalley group of type $\Phi$ over $k$ (all constructed from the same complex simple Lie algebra $\mathcal{L}$). 

Let $f_1:G_{sc}\longrightarrow G_{ad}$ and $f_2:G\longrightarrow G_{ad}$ be the homomorphisms given by $f_1(\exp(t\pi_1(e_\alpha)))=\exp(t\ad(e_\alpha))$ and $f_2(\exp(t\pi(e_\alpha)))=\exp(t\ad(e_\alpha))$ respectively (for all $t\in k,\alpha\in \Phi$). Then $\ker(f_1)=Z(G_{sc})$ and $\ker(f_2)=Z(G)$. 

 For every $\rho\in \Gamma$, let   $\overline{\rho}$ (resp. $\widetilde{\rho}$) be the  automorphism of $G_{ad}$ (resp. $G_{sc}$)  induced by $\rho$ as in \lemref{algebraic}. Then we have the following commutative diagram
	\[\xymatrix{
		G_{sc} \ar[r]^{\widetilde{\rho}} \ar[d]_{f_1}  & G_{sc}\ar[d]^{f_1} \\
		G_{ad} \ar[r]^{\overline{\rho}} &G_{ad}}\]
	and a short exact sequence of groups
	\[\xymatrix{1\ar[r]&Z(G_{sc})\ar[r]&G_{sc}\ar[r]^{f_1}&G_{ad}\ar[r]&1}.\]
	Note that for every $\rho\in \Gamma$, $\widetilde{\rho}(Z(G_{sc}))=Z(G_{sc})$. Since $G_{sc}$ has the $R_{\infty}$-property (by \propref{step1}) and $Z(G_{sc})$ is finite, by virtue of \lemref{ses}(\ref{gtoq}) we have $R(\overline{\rho})=\infty$ for all $\rho\in \Gamma$. Hence by \lemref{ets}, $G_{ad}$ has the $R_{\infty}$-property.
	
	Now consider the short exact sequence 
	\[\xymatrix{1\ar[r]&Z(G)\ar[r]&G\ar[r]^{f_2}&G_{ad}\ar[r]&1}.\]
	Observe that for every $\varphi\in \Aut(G)$, we have $\varphi(Z(G))=Z(G)$ and $G_{ad}$ has the  $R_\infty$-property (proven above). Therefore in view of \lemref{ses}(\ref{qtog}), we conclude that $G$ has the $R_{\infty}$-property. Hence the result.
\end{proof}
Now, we are in a position to prove our main theorem of this paper. 
\begin{theorem}\label{mainthm}
Let $G$ be a connected non-solvable linear algebraic group. Then $G$ satisfies the $R_{\infty}$-property. 
\end{theorem}
\begin{proof}
 First, suppose that $G$ is a connected semisimple algebraic group with root system $\Phi$. Following \secref{agcg}, let $G=G(\pi,k)$ and $G_{ad}=G(\ad, k)$ an adjoint semisimple group of the same type as $G$. Then $G_{ad}$ is a direct product of connected simple groups, i.e., $G_{ad}= G_1^{n_1}\times \cdots \times G_r^{n_r}$, where $G_i$ is simple (for all $1\leq i\leq r$) and $G_i\ncong G_j$ whenever $i\neq j$. Consider the exact sequence $$1\longrightarrow G_2^{n_2}\times\cdots \times G_r^{n_r}\longrightarrow G_{ad}\longrightarrow G_1^{n_1}\longrightarrow 1 .$$
  \noindent Note that $G_2^{n_2}\times\cdots \times G_r^{n_r}$ is invariant under $\Aut(G_{ad})$ and $G_1^{n_1}$ possesses the $R_{\infty}$-property (by \lemref{productgn} and \corref{step2}). Hence by \lemref{ses}(\ref{qtog}), $G_{ad}$ has the $R_{\infty}$-property. If we take the exact sequence $$1\longrightarrow Z(G)\longrightarrow G\overset{f}\longrightarrow G_{ad}\longrightarrow 1$$ where $f:G\longrightarrow G_{ad}$ is given by $f(\exp(t\pi(e_\alpha))=\exp(t\ad(e_\alpha))$ (for all $t\in k$,$\alpha\in \Phi$), we conclude (by \lemref{ses}(\ref{qtog})) that the semisimple group $G$ has the $R_\infty$-property.

 Now let $G$ be a connected non-solvable algebraic group. Then we have the following short exact sequence
\[\xymatrix{1\ar[r]&R(G)\ar[r]&G\ar[r]&G/R(G)\ar[r]&1},\] where $R(G)$ is the radical of $G$.  Note that $R(G)$ is invariant under $\Aut(G)$ and $G/R(G)$ is connected semisimple. Then by the previous paragraph $G/R(G)$ has the $R_\infty$-property and hence, so does $G$ (again by \lemref{ses}(\ref{qtog})).
\end{proof}
\begin{corollary}
Let $G$ be a linear algebraic group such that the connected component $G^0$ of $G$ is non-solvable. Then $G$ satisfies the $R_{\infty}$-property. 
\end{corollary}
\begin{proof}
 For any linear algebraic group $G$ we have a short exact sequence 
\[\xymatrix{1\ar[r]&G^{0}\ar[r]&G\ar[r]&G/G^{0}\ar[r]&1},\] where $G^{0}$ is the connected component of $G$ containing the identity and $G/G^0$ is finite. Observe that $G^{0}$ is invariant under every element of $\Aut(G)$ and by \thmref{mainthm}, $G^0$ has the $R_{\infty}$-property. Therefore, in view of  \lemref{ses}(\ref{ntog}), $G$ possesses the $R_{\infty}$-property.
\end{proof}
\begin{corollary}
Let $G$ be a non-toral reductive algebraic group. Then $G$ has the $R_{\infty}$-property.
\end{corollary}
\begin{proof}
 The proof follows from \thmref{mainthm} since $G/R(G)$ is semisimple.
\end{proof}
We conclude this section with the following two results related to certain classes of solvable groups.  
\begin{proposition}\label{solvable}
	Let $G$ be a connected solvable algebraic group and $T$ a maximal torus of $G$. Suppose that $\varphi(T)=T$ implies $R(\varphi|_{T})=\infty$ for all  $\varphi\in\Aut(G)$. Then $G$ has the $R_{\infty}$-property. 
\end{proposition}
\begin{proof}
	The connected solvable group $G$ can be decomposed as $G=G_u\rtimes T$, where $G_u$ is the subgroup consisting of all the unipotent elements of $G$ and $T$ is a maximal torus. 
	Let $\varphi\in \Aut(G)$. Then there exists $g\in G$ such that $g\varphi(T)g^{-1}=T$, i.e., $\Int_g\circ \varphi(T)=T$. Let $\psi:=\Int_g\circ \varphi$. 
	We will write $\psi|_{G_u}=\psi_1$ and $\psi|_{T}=\psi_2$.  
	Therefore $$[(1,t)]_{\psi}=\{(ast\cdot(\psi_2(s)^{-1}\cdot\psi_1(a)^{-1}), st\psi_2(s)^{-1})\mid a\in G_u, s\in T\}.$$
	Note that $\psi(G_u)=G_u$. Since $R(\psi_2)=\infty$ then $R(\psi)=\infty$. Thus, by \lemref{inner}, $R(\varphi)=\infty$. Hence $G$ has the $R_{\infty}$-property.
\end{proof}

\begin{proposition}\label{unipo}
	Let $G$ be a semisimple algebraic group and $U$ a maximal unipotent subgroup of $G$. Then there exists an automorphism $\varphi\in \Aut(U)$ such that $R(\varphi)=1$.
\end{proposition}

\begin{proof}

By virtue of \cite[Theorem 10.1]{St}, it is enough to find an automorphism $\varphi$ of $U$ such that the fixed point subgroup $U^\varphi$ is finite.

Following \secref{agcg} and Notation \ref{not}, let $G=G(\pi, k)$. We may take $U=\{ \prod_{i=1}^r x_{\alpha_i}(s_i)\mid s_i\in k, \alpha_i\in \Phi^+\}$ and let $T=\{\prod_{i=1}^n h_{\beta_i}(t_i)\mid\alpha\in \Delta, t_i\in k^{\times}\}$ be the maximal torus. 

 \noindent \textbf{Claim:} There exists $t_1,\ldots,t_n\in k^\times$ such that $\prod_{j=1}^nt_j^{<\alpha_i, \beta_j>}-1\neq 0$ for all $i=1,\ldots,r$.
 
 \noindent\emph{Proof of claim:}  Let $f_i(X_1,\ldots,X_n)=\prod_{j=1}^nX_j^{<\alpha_i, \beta_j>}-1\in k(X_1,\ldots,X_n)$ ($i=1,\ldots,r$). If we fix $t_2,\ldots,t_n\in k^\times$, then we have $F_i(X_1)=f_i(X_1,t_2,\ldots,t_n)=b_iX_1^{<\alpha_i,\beta_1>}-1\in k(X_1) (i=1,\ldots,r)$, where $b_i\in k$. Note that the zero set $\mathcal{V}(F_i)$ of each $F_i$ is finite. Hence there exists  a non-zero element $t_1\in k\setminus \bigcup\limits_{i=1}^r\mathcal{V}(F_i)$. Therefore  $f_i(t_1,\ldots,t_n)\neq 0$ for all $i=1,\ldots,r$. Hence, the claim.
 
 Now let $h=\prod_{i=1}^nh_{\beta_i}(t_i)\in T$ (where $t_1,\ldots,t_n$ are as in the claim above) and consider the automorphism $\varphi:U\longrightarrow U$ defined by $\varphi(u)=huh^{-1}$ for all $u\in U$. We know that $h_\beta(t)x_\alpha(s)h_\beta(t)^{-1}=x_\alpha(t^{<\alpha,\beta>}s)$ for all $\alpha\in \Phi^+, \beta\in \Delta, t\in k^\times, s\in k $. So if  $y=\prod_{i=1}^rx_{\alpha_i}(s_i)\in U^\varphi$ then we have $y=\varphi(y)=hyh^{-1}=\prod_{j=1}^n h_{\beta_j}(t_j)\prod_{i=1}^r x_{\alpha_i}(s_i) \prod_{j=1}^n h_{\beta_j}(t_j)^{-1}
=\prod_{i=1}^rx_{\alpha_i}(\prod_{j=1}^nt_j^{<\alpha_i, \beta_j>}s_i)$. Therefore by the uniqueness of the factorisation $y=\prod_{i=1}^rx_{\alpha_i}(s_i)$, we have $(\prod_{j=1}^nt_j^{<\alpha_i, \beta_j>}-1)s_i=0$ (for all $i=1,\ldots,r$).  So by our choice of $t_1,\ldots,t_n$, we have $s_i=0$ (for all $i=1,\ldots,r$) which implies that $y=1$. Thus $U^\varphi=\{1\}$ and hence, $R(\varphi)=1$.
\end{proof}
\section{Examples}\label{example}
In this section we compute $R(\varphi)$ for certain classes of algebraic groups. Example \ref{eg3} shows that the sufficient condition proven in \thmref{mainthm} is not necessary.
\begin{enumerate}
\item Let $\mathrm{D}_n(k)=\{\diag(t_1,\ldots, t_n)\mid t_i\in k^{\times}\}$ ($n\geq 1$) be the diagonal subgroup of $\GL_n(k)$. 
\begin{enumerate}\label{diag}
\item\label{diag1} Let $\varphi$ be an automorphism of $\mathrm{D}_n(k)$ given by $$\varphi(\diag(t_1,\ldots, t_n))=\diag(t_1^{-1},\ldots, t_n^{-1}).$$ Then $\mathrm{D}_n(k)=[I_n]_{\varphi}$. Thus $\mathrm{D}_n(k)$ does not satisfy the $R_{\infty}$-property.
\item For $r\in \mathbb{N}$, let $\varphi_r$ be an automorphism of $\mathrm{D}_n(k)$ given by $$\varphi_r(\diag(t_1, t_2, \ldots, t_n))=\diag(t_n, t_1, t_2, \ldots, t_{n-2}, t_{n-1}t_n^{-r}).$$ Then $\mathrm{D}_n(k)=[I_n]_{\varphi_r}$ for all $r\in \mathbb{N}$.
\end{enumerate} 
\item Let $\varphi=\varphi_1\times \varphi_2$ be the automorphism of $\mathbb{G}_a\times \mathbb G_m$, where $\varphi_1$ is an automorphism of $\mathbb G_a$ given by $\varphi_1(x)=\alpha x$ for $\alpha\in k^{\times}\setminus \{1\}$ and $\varphi_2$ is an automorphism of $\mathbb{G}_m$ given by $\varphi_2(x)=x^{-1}$. Then $R(\varphi)=R(\varphi_1)R(\varphi_2)=1$. Therefore $\mathbb{G}_a\times \mathbb G_m$ does not satisfy the $R_{\infty}$-property. 
\item\label{unipotent} Let $G=\U_n(k)$ ($n\geq 1$) be the group of all upper triangular unipotent matrices of $\GL_n(k)$, $d:=\diag(t_1, t_2, \ldots, t_n)\in \mathrm{D}_n(k)$ such that $t_i\neq t_j$ for all $i\neq j$ and  define an automorphism $\varphi_d: \U_n(k)\longrightarrow \U_n(k)$ by $\varphi_d(g)=dgd^{-1}$ for all $g\in \U_n(k)$. 

\noindent
\textbf{Claim:} $R(\varphi_d)=1$, i.e. $g\sim_{\varphi_d}I_n$ for all $g\in\U_n(k)$.

\noindent\emph{Proof of claim:} Let $g=(g_{ij})\in \U_n(k)$. It suffices to find some $y\in \U_n(k)$ such that $g=y^{-1}\varphi_d(y)$ or equivalently,
\begin{equation}\label{solution}
yg=dyd^{-1}
\end{equation}
 For $1\leq i<j\leq n$ let us consider the following system of equations in the variables $x_{ij}$
\begin{equation}\label{eq}
(t_it_j^{-1}-1)x_{ij}=g_{ij}+\displaystyle\sum_{i< k<j}x_{ik}g_{kj}
\end{equation}

Now, by the assumption on $d$, we note that $(t_it_j^{-1}-1)\neq 0$ for all $i\neq j$. Therefore, it is clear that Equation \eqref{eq} admits a unique solution say $y_{ij}\quad (1\leq i<j\leq n)$. If we set $y_{ii}=1$ for all $1\leq i\leq n$ and $y_{ij}=0$ for all $i>j$, then $y=(y_{ij})$ satisfies Equation \eqref{solution}. Hence, the proof. \qed 

(See \cite[Theorem 2]{timur19}).
\item Let $G=\mathrm{D}_l(k)\times \U_n(k)$ ($l, n\geq 1$). The map $\psi:G\longrightarrow G$ defined by $\psi(s,u)=(\varphi(s),\varphi_d(u))$  (for all $(s,u)\in G$) is an automorphism (where $\varphi$ is as in Example \ref{diag1} and $d,\varphi_d$ as in Example \ref{unipotent}). Observe that $R(\psi)=R(\varphi)R(\varphi_d)=1$. 
\item\label{eg3} The standard Borel subgroup $\B_2(k)=\Big\{\begin{pmatrix}a&b\\0&a^{-1}\end{pmatrix}\mid a\in k^{\times}, b\in k\Big\}$ of $\SL_2(k)$ satisfies $R_\infty$-property. Note that $\B_2(k)=U\rtimes T\cong \mathbb{G}_a\rtimes \mathbb{G}_m$, where $T=\Big\{\begin{pmatrix}a&0\\0&a^{-1}\end{pmatrix}\mid a\in k^{\times}\Big\}\cong\mathbb{G}_m$ and $U=\Big\{\begin{pmatrix}
1&y\\0&1
\end{pmatrix}\mid y\in k\Big\}\cong \mathbb{G}_a$. Let $\varphi\in \Aut(\B_2(k))$. Since $T$ and $\varphi(T)$ are two maximal torus then there exists $g\in \B_2(k)$ such that $g\varphi(T)g^{-1}=T$. Then the automorphism $\psi:=\Int_g\circ \varphi$ of $\B_2(k)$ stabilizes $U$ as well as $T$. Hence there exists $\alpha\in k^\times$ such that $\psi\left(\begin{pmatrix}
1&y\\0&1
\end{pmatrix}\right)=\begin{pmatrix}
1&\alpha y\\0&1
\end{pmatrix}$, for all $y\in k$. Now the restriction of $\psi$ on $T$ has two possibilities : either $\psi(\diag(t, t^{-1}))=\diag(t, t^{-1})$ or $\psi(\diag(t, t^{-1}))=\diag(t^{-1}, t)$ for all $t\in k^{\times}$.

\noindent
\textbf{Claim:} $\psi(\diag(t, t^{-1}))=\diag(t, t^{-1})$ for all $t\in k^{\times}$.

\noindent \emph{Proof of claim:} If possible let $\psi(\diag(t, t^{-1}))=\diag(t^{-1}, t)$ for all $t\in k^{\times}$. Then we have $\psi \left(\begin{pmatrix}a&b\\0&a^{-1}\end{pmatrix}\right)=\psi\left(\begin{pmatrix}
a&0\\0&a^{-1}
\end{pmatrix}\right) \psi \left(\begin{pmatrix}
1&ba^{-1}\\0&1
\end{pmatrix}\right)$ $=\begin{pmatrix}
a^{-1}&0\\0&a
\end{pmatrix}\begin{pmatrix}
1&\alpha ba^{-1}\\0&1
\end{pmatrix}=\begin{pmatrix}a^{-1}&\alpha a^{-2}b\\0&a\end{pmatrix}$ on $\B_2(k)$. But clearly, this is not a homomorphism. Hence the claim. 

Therefore $\psi=\Id$ on $T$. 
Thus, $R(\psi|_{T})=R(\Id)=\infty$. Then, in view of  \propref{solvable}, $R(\psi)=\infty$. By \lemref{inner}, we have $R(\varphi)=R(\Int_g\circ \varphi)=R(\psi)=\infty$.
 Hence the group $\B_2(k)$  has the $R_{\infty}$-property although it is solvable. 
\end{enumerate}
\medskip
\textbf{Acknowledgements:} 
The authors would like to thank Maneesh Thakur, Shripad Garge and Timur Nasybullov for their valuable suggestions and comments on this work. We also thank the reviewers for some extremely crucial comments and suggestions which improved the exposition considerably.

\end{document}